\documentclass[12pt]{amsart}
\usepackage{amsmath,enumerate}
\usepackage{amssymb}
\usepackage[top=50pt, bottom=50pt, left=40pt, right=40pt]{geometry}
\newtheorem{theorem}{Theorem}

\newtheorem{defn}[theorem]{Definition}
\newtheorem{lemma}[theorem]{Lemma}

\newtheorem*{remark*}{Remark}

\newcommand{\C}{\Bbb C}

\newenvironment{proof*}{\vskip 2mm\noindent {}}{\hfill $\Box$ \vskip 2mm}

\title{On the non-existence of limit E-Brody curves}

\begin{document}

\author{Tran Duc-Anh}
\address{Department of Mathematics \\
Hanoi National University of Education\\
136 Xuan Thuy St., Hanoi, Vietnam} \email{ducanh@hnue.edu.vn}


\begin{abstract}
We give a simple proof of the non-existence of limit E-Brody curves,
in the sense of Do Duc Thai, Mai Anh Duc and Ninh Van Thu, for a
class of manifolds including $\mathbb{C}^n$ and
$(\mathbb{C}^{\ast})^2$ which were studied by these authors in
\cite{Do DT et al}, by constructing a suitable holomorphic
interpolation function.
\end{abstract}

\maketitle
\section{Introduction}
Non-normal families of holomorphic mappings and Brody curves (i.e. a holomorphic curve with bounded derivatives) are closely related (cf. \cite{Brody, Zal}). Concerning these topics, in \cite{Do DT et al}, the authors proved the following results.

\begin{theorem}[Theorem 1.6 of \cite{Do DT et al}]\label{theorem 1.6}
$\mathbb C^n~(n\geq 2)$ is not of $E$-limit type for any  length function $E$ on $\mathbb C^n$.
\end{theorem}

\begin{theorem}[Theorem 1.7 of \cite{Do DT et al}] \label{theorem}$(\mathbb C^*)^2$ is not of $ds^2_{FS}$-limit type, where $ds^2_{FS}$ is the Fubini-Study metric on $\mathbb P^2(\mathbb C)$.
\end{theorem}

Their proof makes use of a result of J. Winkelmann and some quite
complicated techniques of extraction of subsequences. So in this
note, we try to give another short and slightly more general proof.

First of all, we recall some definitions to explain the content of their theorems.

\begin{defn}
Let $X$ be a complex manifold with a hermitian metric $E.$ A
holomorphic curve $f: \mathbb C\to X$ is said to be an
$\mathbf{E}$\textbf{-Brody curve} if its derivative is bounded,
i.e.,  $|f'(z)|_E\leq c$ for every  $z\in\mathbb C$ where $c$ is a
constant positive constant.
\end{defn}

A length function is a more general notion of a metric, but it will not concern us in this note, so we refer the reader to their paper \cite{Do DT et al} for the definition. Note that if we write $E_p(\vec{v}),$ it means the length of the tangent vector $\vec{v}$ at the point $p$ with respect to the metric $E.$ If the point $p$ is well understood, $|\vec{v}|_E$ is sufficient.

\begin{defn} Let $X$ be a complex manifold with a hermitian metric $E$. The complex manifold $X$ is said to be of \textbf{$\mathbf{E}$-limit type}
 if $X$ satisfies the following:

For each non-normal family $\mathcal{F}\subset \mathrm{Hol}(\Delta,
X),$ where $\Delta$ is a domain in $\C$ and $\mathrm{Hol}(\Delta,
X)$ the set of all holomorphic mappings from $\Delta$ into $X,$ such
that $\mathcal{F}$ contains no compactly divergent sequence, then
there exist sequences $\{p_j\}\subset \Delta$ with $p_j\to
p_0\in\Delta$ as $j\to\infty$, $\{f_j\} \subset \mathcal{F},
\{\rho_j\}\subset \mathbb R$ with $\rho_j>0$ and $\rho_j \to 0^+$ as
$j\to \infty$ such that
$$
g_j(\xi):=f_j(p_j+\rho_j\xi), \xi\in \mathbb  C,
$$
converges uniformly on any compact subsets of $\mathbb C$ to a non-constant $E$-Brody curve $g: \mathbb C\to X$.
\end{defn}

For the convenience of the reader, we recall the definition of a normal family and compact divergence, which the reader can find in \cite{Do} or\cite{Do DT et al}.

\begin{defn}A family $\mathcal{F}\subset \mathrm{Hol}(\Delta, X)$ is said to be \textbf{normal} if, for each sequence $\{f_j\}_{j=1}^{\infty}$ in $\mathcal{F},$ there exists a subsequence which converges uniformly on compact subsets of $\Delta.$ A family of mappings which is not normal is called a \textbf{non-normal family}.\end{defn}

\begin{defn} A sequence $\{f_j\}_{j=1}^{\infty}$ in $\mathrm{Hol}(\Delta, X)$
is said to be \textbf{compactly divergent} if, for all compacts
$K\subset \Delta$ and $L\subset X,$ there exists $j_0$ such that,
for $j\geq j_0,$ we have $f_j(K)\cap L = \emptyset.$
\end{defn}

\section{Non-existence of limit Brody curves}
Now we proceed to the main content of this note. We will prove a slight generalization of  their Theorem \ref{theorem} (or Theorem 1.7 in their original paper \cite{Do DT et al}) stated as follows.

\begin{theorem}[Main result]\label{main} Let $X$ be a complex manifold
 which contains an entire curve, i.e. there exists a non-constant holomorphic
 curve $f\colon \C\to X.$ Then both $\C\times X$ and $\C^{\ast}\times X$
 are not of $E-$limit type for any hermitian metric $E$ on $\C\times X$ or
 $\C^{\ast}\times X$ respectively.
 \end{theorem}

In the proof of the theorem, we make use of two lemmas. The first is stated without proof (cf. page 299, chapter 15 \cite{Rudin}).

\begin{lemma}\label{lemma}Let $c_n>0$ for $n\in\mathbb{N}.$ The following are equivalent.

(a) $\prod_{n=1}^{\infty}(1+c_n) <\infty,$
(b) $\sum_{n=1}^{\infty}c_n < \infty.$

Moreover, (a) and (b) are equivalent to the following

(c) $\prod_{n=1}^{\infty}(1-c_n) >0$ if we suppose in addition $0<c_n<1$ for every $n.$
\end{lemma}

The second is the main lemma, which is the key of the proof of the Main result.

\begin{lemma}\label{mainlemma}Let $\{\alpha_j\}_{j=1}^{\infty}$ be a sequence of pairwise distinct nonzero complex numbers such that $\sum_{j=1}^{\infty}\frac{1}{|\alpha_j|}<\infty.$ Then, for all complex numbers $p_j$ and $k_j$ with $1\leq j\in \mathbb{N},$ there exists a holomorphic function $g\colon \C\to \C$ which satisfies the following interpolation conditions: $g(\alpha_j)=p_j$ and $g'(\alpha_j)=k_j$ for every $1\leq j\in \mathbb{N}.$ \end{lemma}

\begin{proof}By making use of Lemma \ref{lemma}, the condition $\sum_{j=1}^{\infty}\frac{1}{|\alpha_j|}<\infty$ implies  the series $\prod_{j=1}^{\infty}\left(1-\frac{z}{\alpha_j}\right)^2$ converges to a non-constant entire function $h$ whose only zeros are the $\alpha_j$'s at which $h$ has zero derivative.

Denote by $\mathcal{O}$ the sheaf of holomorphic functions on $\C$ and by $\mathcal{I}$ the sheaf of homolomorphic functions on $\C$ which obtain $\alpha_j$ as zeros of order bigger than 1. Then we have $\mathcal{I} = h\cdot \mathcal{O},$ and therefore $\mathcal{I}$ is a coherent sheaf (cf. page 130, Proposition 8, chapter IV\cite{GunningRossi}).

Now we make use of a classical trick as in the Mittag-Leffler theorem. Locally around each point $\alpha_j,$ we always find a holomorphic function $f_j$ such that $f_j(\alpha_j) = p_j$ and $f_j'(\alpha_j) = k_j.$ It means we always find a cochain $\{(U_j,f_j)\},$  where $U_j$ is an open neighborhood of $\alpha_j$ for $j\geq 1,$ in the \v{C}ech cochain complex $C^0(\mathcal{U},\mathcal{O}),$ where $\mathcal{U} = \{U_j~:~j\geq 1\}$ is an open covering of $\C$ and $f_j(\alpha_j) = p_j,$ $f_j'(p_j) = k_j$ for $j\geq 1.$

Denote by $\delta$ the differential of the \v{C}ech complex. Then $\delta\{(U_j,f_j)\}$ is a cocycle in $Z^1(\mathcal{U},\mathcal{I})$ up to a refinement of $\mathcal{U}.$ But the sheaf $\mathcal{I}$ is coherent and  the underlying space $\C$ is Stein, therefore its cohomology $H^1(\C,\mathcal{I})$ vanishes.

It means that we can find a cochain $\{(U_j,g_j)\}\in C^0(\mathcal{U},\mathcal{I})$ such that $\delta\{(U_j,f_j)\}=\delta\{(U_j,g_j)\}.$ It implies that $f_j - g_j = f_i-g_i$ in $U_j\cap U_i$ for $j\neq i,$ then they define a global function $g\colon \C\to \C$ which satisfies the above interpolation conditions.
\end{proof}

\begin{proof}[Proof of Theorem \ref{main}] We prove only the case of $\C^{\ast}\times X,$ the proof of the remaining case is analogous.

 Since $f$ is non-constant and holomorphic, there exists a sequence of nonzero complex numbers $\{p_j\}_{j=1}^{\infty}$ in $\C$ such that $f'(p_j)\neq 0$ for every $j.$

Suppose $g\colon \C\to \C$ is a holomorphic function such that $g(\alpha_j) = p_j$ and $g'(\alpha_j)=k_j$ for $j\geq 1.$ This function always exists by Lemma \ref{mainlemma}. We compute the length of the tangent vector to the curve $\C\ni z\mapsto (e^z, (f\circ g)(e^z))$ at the point $z=q_j$ where $e^{q_j}=\alpha_j.$

We have \begin{align*}
E_{(\alpha_j, f(p_j))}(\alpha_j, f'(p_j)g'(\alpha_j)\alpha_j) &= E_{(\alpha_j, f(p_j))}(\alpha_j, f'(p_j)k_j\alpha_j)\\ & \geq |k_j|E_{(\alpha_j, f(p_j))}(0, f'(p_j)\alpha_j)-E_{(\alpha_j, f(p_j))}(\alpha_j,0).\end{align*}

So if we choose $$k_j = j\cdot\left(\frac{1}{E_{(\alpha_j, f(p_j))}(0, f'(p_j)\alpha_j)}+E_{(\alpha_j, f(p_j))}(\alpha_j,0)\right),$$ the entire curve $\C\ni z\mapsto (e^z, (f\circ g)(e^z)),$ which is from now on denoted by $F,$ is not an $E-$Brody curve.

Now consider the sequence of
$f_n\in\mathrm{Hol}(\Delta,\C^{\ast}\times X)$ defined by the
formula $f_n(z)= F(nz).$ This sequence contains no compactly
divergent subsequence since $f_n(0) = F(0)$ for every $n.$ So if
$\C^{\ast}\times X$ were of $E-$limit type, then we could extract a
subsequence, which is still denoted by $f_j,$ such that there exist
points $a_j\to a_0\in \Delta$ and $\rho_j>0$ which tends to $0$ with
following property: the sequence of mappings $$\C\ni \xi\mapsto
f_j(a_j + \rho_j\xi)$$ converges uniformly on compact subsets to a
non-constant entire curve denoted by $G.$

The first coordinate of  the mapping $\C\ni \xi\mapsto f_j(a_j + \rho_j\xi)$ is $e^{ja_j + j\rho_j\xi}$ which is a nonvanishing holomorphic function. Denote by $L$ the limit function of this sequence, then  $G = (L, (f\circ g)(L)).$ We deduce that $L$ is non-constant because of the non-constancy of $G.$ By Hurwitz's theorem, $L$ is also everywhere non-vanishing.

We have $$j\rho_j = \frac{\frac{d}{d\xi}e^{ja_j +
j\rho_j\xi}}{e^{ja_j + j\rho_j\xi}}$$ which, as a sequence of
functions, converges uniformly on compact subsets to a holomorphic
function which assumes only real values. Therefore $j\rho_j$ tends
to a nonzero number $B.$ We also have $e^{ja_j}$ converging to some
nonzero number $e^A.$

Thus $L$ is of the form $L(\xi) = \exp(A+B\xi)$ and it implies $$G(\xi) = (\exp(A+B\xi), (f\circ g)(\exp(A+B\xi))) = F(A+B\xi).$$

By the construction of $F,$ this gives a contradiction to the supposed $E-$limit type of $\C^{\ast}\times X.$
 \end{proof}

\vskip.5cm {\bf Acknowledgements.} This note is a by-product during
the author's stays at the Universit\'e Paul Sabatier thanks to the
financial aids of LIA Formath Vietnam and the Laboratoire Emile
Picard, Institut de Math\'ematiques de Toulouse. I would like to
thank Prof Do Duc Thai and Dr Ninh Van Thu for having pointed out an
error in my previous manuscript two years ago. I would also like to
thank Prof Pascal J. Thomas for having found these financial aids to
make possible my stays at the university. Finally, I would like to
thank the referee for his/her comments on the presentation of my
manuscript which leads to this present version.

\end{document}